\documentclass[11pt,reqno]{amsart}
\usepackage{amsmath,amssymb}
\usepackage{mathabx}
\usepackage{mathdots}
\usepackage{graphicx,caption,subfig}
\usepackage{color,soul}
\usepackage{hyperref}
\usepackage{enumerate}
\usepackage[colorinlistoftodos]{todonotes}
\usepackage{enumitem}
\usepackage[normalem]{ulem}
\usepackage{constants}

\usepackage{doi}

\newconstantfamily{N}{symbol=N}

\usepackage{xcolor}

\usepackage{fullpage}

\newtheorem{theorem}{Theorem}[section]

\newtheorem{definition}{Definition}[section]
\newtheorem{corollary}{Corollary}[section]
\newtheorem{proposition}{Proposition}[section]
\newtheorem{remark}{Remark}[section]

\usepackage{prettyref}
\newrefformat{def}{Definition \ref{#1}}
\newrefformat{thm}{Theorem \ref{#1}}
\newrefformat{prop}{Proposition \ref{#1}}
\newrefformat{lem}{Lemma \ref{#1}}
\newrefformat{cor}{Corollary \ref{#1}}
\newrefformat{sec}{Section \ref{#1}}
\newrefformat{sub}{Subsection \ref{#1}}
\newrefformat{rem}{Remark \ref{#1}}
\newrefformat{ex}{Example \ref{#1}}
\newrefformat{app}{Appendix \ref{#1}}
\newrefformat{subfig}{Fig. \ref{#1}}
\newrefformat{fig}{Fig. \ref{#1}}

\numberwithin{equation}{section}

\DeclareMathOperator{\Arg}{Arg}

\DeclareMathOperator{\diag}{diag}

\newcommand{\w}{\omega}
\newcommand{\ii}{\imath}

\newcommand{\CC}{\mathbb{C}}
\newcommand{\NN}{\mathbb{N}}
\newcommand{\RR}{\mathbb{R}}
\newcommand{\TT}{\mathbb{T}}

\renewcommand{\vec}[1]{\boldsymbol{\mathrm{#1}}}
\newcommand{\xv}{\mathcal{X}} 
 
\newcommand{\xiv}{\boldsymbol{\xi}}

\newcommand{\cv}{\vec{c}} 

\newcommand{\VV}{\mathbf{V}}

\newcommand{\GG}{\mathbf{G}}

\newcommand{\sep}{\Delta}
\newcommand{\dst}{d}
\newcommand{\Pcx}{P_{\cv,\xv}}

\newcommand{\exs}[1][\ell]{\mathfrak{I}_{#1}}

\newcommand{\SRF}{\operatorname{SRF}}
\makeatletter
\DeclareRobustCommand*\cal{\@fontswitch\relax\mathcal}
\makeatother

\begin{document}
\title{Single-exponential bounds for the smallest singular value
of {V}andermonde matrices in the sub-Rayleigh regime}

\author[D.Batenkov]{Dmitry Batenkov} \address{Department of Applied
  Mathematics, School of Mathematical Sciences, Tel Aviv University,
  P.O. Box 39040, Tel Aviv 6997801, Israel}
\email{dbatenkov@tauex.tau.ac.il} \thanks{D.B. was supported by the
  ISRAEL SCIENCE FOUNDATION (grant No. 1793/20)}

\author[G. Goldman]{Gil Goldman} \address{Department of Mathematics,
  The Weizmann Institute of Science, Rehovot 76100, Israel}
\email{gil.goldman@weizmann.ac.il} \thanks{}

  \date{}

\begin{abstract}
  Following recent interest by the community, the scaling of the
  minimal singular value of a Vandermonde matrix with nodes forming
  clusters on the length scale of Rayleigh distance on the complex
  unit circle is studied. Using approximation theoretic properties of
  exponential sums, we show that the decay is only single exponential
  in the size of the largest cluster, and the bound holds for
  arbitrary small minimal separation distance. We also obtain a
  generalization of well-known bounds on the smallest eigenvalue of
  the generalized prolate matrix in the multi-cluster
  geometry. Finally, the results are extended to the entire spectrum.
\end{abstract}

\maketitle

\section{Introduction}\label{sec.intro}

For an ordered set of distinct nodes $\xv=\{ x_1,\ldots,x_s\}$ with
$x_j \in (-\pi,\pi]$, and $N \geq s-1$, consider the $(N+1) \times s$
Vandermonde matrix
\begin{equation} \VV_{N}(\xv) := \left[e^{\ii k x_j}\right]_{0
    \le k \le N}^{1 \le j \le s}.
\end{equation}

This class of matrices is the subject of numerous recent
investigations in the applied harmonic analysis community,
e.g. \cite{aubel_vandermonde_2017,batenkov2018a,batenkov2019,ferreira_superresolution_1999,kunis2018,kunis2019a,li_stable_2017,li2020,liao_music_2016,moitra_super-resolution_2015}. While
interesting in their own right, the spectral properties of $\VV_{N}$
are closely related to the problem of super-resolution (SR) under
sparsity constraints, which also received a lot of attention in recent
years
\cite{batenkov2019a,candes_towards_2014,demanet_recoverability_2014,donoho_superresolution_1992}. In
the SR context, the smallest singular value
$\sigma_{\min}(\VV_N):=\min_{\cv\in\CC^s,\|\cv\|_2=1}\|\VV_N \cv\|_2$
controls the limit of stable recovery of a superposition of Dirac
masses supported on $\xv$ from its first $N+1$ Fourier coefficients,
while the singular subspaces play a major role in various SR
algorithms (e.g. MUSIC and ESPRIT)
\cite{fannjiang_compressive_2016,li_stable_2017,li2020,liao_music_2016}.

Let $\Delta$ denote the minimal separation (in the wrap-around sense)
between the elements of $\xv$. With $s$ fixed, two distinct asymptotic
regimes are known:
\begin{enumerate}
\item When $N\Delta \gtrapprox O(1)$, the matrix $\VV_N$ is
  well-conditioned, and $\sigma_{\min}(\VV_N) = O(\sqrt{N})$.
\item When $N\Delta\ll 1$, $\sigma_{\min}(\VV_N)$ can be as small as $O\left(\sqrt{N}(N\Delta)^{s-1}\right)$.
\end{enumerate}

The well-conditioned case 1) has been studied in
\cite{aubel_vandermonde_2017,ingham_trigonometrical_1936,moitra_super-resolution_2015,montgomery_hilberts_1974,negreanu_discrete_2006},
by various tools from harmonic analysis and analytic number
theory. The separation condition $\Delta\gtrapprox {1\over N}$ plays a
major role in the analysis of the convex relaxations of the SR problem
\cite{candes_super-resolution_2013,candes_towards_2014}.

Case 2) corresponds to the so-called ``sub-Rayleigh'' regime,
where $N=2\pi\Delta^{-1}$ is precisely the Rayleigh resolution limit. The
possibility to resolve closely spaced point sources from low-frequency
measurements with arbitrary precision was already established by
G.~de~Prony in 1795 \cite{prony1795}\footnote{English translation of
  the original Prony's paper can be found in \cite{auton1981}.},
providing the symbolic-algebraic basis for many other reconstruction
algorithms that followed. However, without additional prior
information regarding the geometry of $\xv$, the sensitivity to noise
(``condition number'') of the SR problem in the sub-Rayleigh
regime may be as large as $\SRF^{2s-1}$, where $\SRF:=(N\Delta)^{-1}$
is the ``super-resolution factor''. This quickly becomes prohibitive
already for moderate values of $s$. The exponent $2s-1$
corresponds to the worst-case scenario where all the nodes of $\xv$
are clustered together and approximately equispaced, e.g.
$x_{j+1} = x_j + \Delta_j$ with $\Delta_j \approx \Delta$,
$j=1,\dots,s-1$.

When $\Delta_j=\Delta$ for all $1\leq j \leq s-1$, $\VV_{N}$ is a
contiguous submatrix of the DFT matrix (also known as the ``prolate
matrix'' in the literature \cite{varah_prolate_1993}), and the scaling of
$\sigma_{\min}(\VV_N)$, in the asymptotic regime $N\to\infty,\;N\Delta <c$
directly follows from the ``Bell Labs theory'' of the spectral
concentration problem \cite{slepian_prolate_1978} (see also
\cite{barnett2020}). See \prettyref{sec:comparison} for further
discussion.

Now suppose that an a-priori information is available, according to
which only a small number of nodes can be clustered, with the
different clusters separated by $\theta\gtrapprox {1\over N}$ (see
\prettyref{def:partial.cluster} below). It has been recently shown by
several groups that in this case,
$\sigma_{\min}(\VV_N)\asymp \sqrt{N}(N\Delta)^{\ell-1}$ where $\ell$
is the largest multiplicity of any such cluster. Accordingly, the SR
condition number will scale as $\SRF^{2\ell-1}$
\cite{batenkov2018a,batenkov2019,batenkov2019a,diederichs2019,kunis2018,kunis2019a,li_stable_2017}.

The scaling of the proportionality constants in the order estimates
above, in particular, their dependence on $\ell$ and $s$, are a
subject of ongoing research. This question is of importance in the
regime where $\ell \ll s$, so that the factor $\SRF^{2\ell-1}$ is
significantly smaller than $\SRF^{2s-1}$.

In this paper we prove a single-exponential in $\ell$ and linear in $s$ lower bound for $\sigma_{\min}(\VV_N(\xv))$ in the multi-cluster geometry (\prettyref{thm:main-theorem}), of the form
\begin{equation}\label{eq:main-lower-bound}
	\sigma_{\min}(\VV_N(\xv)) \geq c_1 \sqrt{N} (N\Delta/c_2)^{\ell-1},
\end{equation}
where $c_1,c_2$ are absolute constants, independent of
$\ell,s,N,\Delta$, (and in fact, $c_2=32\pi e$), holding whenever
\begin{equation}\label{eq:main-bound-condition}
	c_3(\ell) s/\theta \leq N \leq c_4(\ell)/(\tau\Delta s).
\end{equation}

Relative to prior works on the subject, in particular
\cite{kunis2019a,li_stable_2017} 
 (see \prettyref{sec:comparison}
below), our single-exponential in $\ell$ bound
\eqref{eq:main-lower-bound} holds for a fixed $N$ and \emph{all}
sufficiently small $\Delta$.  Applying a simple
limiting argument, in \prettyref{thm:prolate} we also generalize
Slepian's bound for the smallest eigenvalue of the prolate matrix (see
above) in the non-equispaced multi-cluster case.

The main technical contribution of this paper is a new method of proof
of the bound \eqref{eq:main-lower-bound} for a single cluster (Theorem
\ref{thm:single.cluster}), which was previously shown in this setting
in \cite[Example 4.8]{kunis2019a} (again, see details in
\prettyref{sec:comparison}). The proof is based on applying the
classical Turan's inequality for exponential sums and Salem’s
inequality to the analysis of stability of Vandermonde matrices with
nodes on the unit circle. The extension of this result to Vandermonde
matrices with multiple sets of clustered nodes separated by
$\theta \gtrsim {1\over N}$, is done by invoking our recent result
\cite[Theorem 2.2]{batenkov2019}, which, in turn, shows that the
column subspaces in $\CC^{N+1}$ corresponding to each cluster are
nearly orthogonal.

Our results can easily be extended to show single-exponential scaling
for \emph{all the singular values} of $\VV_N$ (resp. eigenvalues of
the prolate matrix). We present some details of these extensions in
\prettyref{sec:entire-spectrum}, however for the sake of brevity we do
not provide the full derivations.

\section{Main results}\label{sec:main.results}

\begin{definition}[Wrap-around distance]\label{def.warp.around}
  For $x,y \in \RR$, we define the wrap-around distance
  $$
  \dst(x,y):= \bigl|\Arg \exp{\imath (x-y)}\bigr| = \bigl|x-y \mod (-\pi,\pi]\bigr| \in \left[0,\pi\right],
  $$ 
  where for $z\in {\mathbb C} \backslash \{0\}$, $\Arg(z)$ is the
  principal value of the argument of $z$, taking values in
  $\left(-\pi,\pi\right]$.
\end{definition}

We denote by $\TT=\RR\mod (-\pi,\pi]$ the periodic interval of length
$2\pi$.

\begin{definition}[Single cluster configuration]\label{def:single.cluster}
  The node set $\xv=\left\{ x_1,\dots,x_\ell\right\} \subset (-\pi,\pi]$
  is said to form a $(\Delta,\ell,\tau)_{\TT}$-cluster, for some $\ell-1\leq\tau\le{\pi\over\Delta}$, if
    $$
    \forall x,y \in \xv, x\neq y: \quad \Delta \le \dst(x,y) \le \tau\Delta.
    $$

\end{definition}

Below we write $C_k(\ell)$, for some indexes $k=1,\ldots$, to indicate
constants that depend only on $\ell$.

Our first main result is the following.

\begin{theorem}\label{thm:single.cluster}
  Let $\xv$ form a $(\Delta,\ell,\tau)_{\TT}$-clustered configuration.  Then
  there exist a constant $\Cl{single.cluster.N}(\ell)$ and absolute
  constants $\Cl{single.cluster.lower}=32\pi e, \Cl{single.cluster.mult}$,
  such that for any $N$ satisfying
  $\Cr{single.cluster.N}(\ell)\le N \le \frac{2\pi}{\tau \Delta}$,
  \begin{equation}\label{eq:single-min-lb}
  \sigma_{\min}(\VV_N(\xv))) \ge
  \Cr{single.cluster.mult}\sqrt{N}\biggl(\frac{N
    \Delta}{\Cr{single.cluster.lower}}\biggr)^{\ell -1}.
  \end{equation}
\end{theorem}

\begin{definition}[Multi-cluster configuration, periodic interval]\label{def:partial.cluster}
  The node set $\xv=\left\{ x_1,\dots,x_s\right\} \subset (-\pi,\pi]$
  is said to form a $\left(\Delta,\theta,s,\ell,\tau\right)_{\TT}$-clustered
  configuration for some $\Delta>0$, $1\leq \ell\leq s$,
  $\ell-1\leq\tau\le{\pi\over\Delta}$ and $\theta > 0$, if for each
  $x_j$, there exist at most $\ell$ distinct nodes
  $$
  \xv^{(j)}=\{x_{j,k}\}_{k=1,\dots,r_j} \subset \xv,\;1\leq
  r_j\leq\ell,\quad x_{j,1}\equiv x_j,
  $$
  such that the following conditions are satisfied:
  \begin{enumerate}
  \item For any $y\in\xv^{(j)}\setminus\{x_j\}$, we have
    $$
    	\Delta\leq d(y,x_j) \leq \tau \Delta.
    $$
  \item For any $y\in\xv \setminus\xv^{(j)}$, we have
    $$
    d(y,x_j) \geq \theta. 
    $$
  \end{enumerate}
\end{definition}

\begin{theorem}\label{thm:main-theorem}
  There exist constants $\Cl{multi.cluster.N.theta}(\ell)$, $\Cl{multi.cluster.N.delta}(\ell)$ and absolute 
  constants $\Cl{multi.cluster.lower.in}=32\pi e, \Cl{multi.cluster.mult}$, 
  such that for any $\xv$ forming a
  $\left(\Delta,\theta,s,\ell,\tau\right)_{\TT}$-clustered configuration and
  $N$ satisfying  $\frac{\Cr{multi.cluster.N.theta}s}{\theta} \le N \le
  \frac{\Cr{multi.cluster.N.delta}}{s\tau\Delta}$,
  \begin{align}
      \label{eq.multi.cluster.bound}
      \sigma_{\min}\left(\VV_N\left(\xv\right)\right) &\geq
			\Cr{multi.cluster.mult}\sqrt{N} \left(\frac{N \Delta}{\Cr{multi.cluster.lower.in}} \right)^{\ell-1}.
  \end{align}
\end{theorem}

\begin{definition}[Generalized prolate matrix]\label{def:prolate}
	Let $\xv=\left\{ x_1,\dots,x_s\right\}\subset\RR$ be a collection of $s$	pairwise distinct points on the real line. We define the generalized prolate matrix as follows:
	$$
	\GG(\xv)=\biggl[ \frac{1}{2}\int_{-1}^1 e^{\imath \omega
		(x_j-x_k)}\biggr]_{j,k=1}^s \in \RR^{s\times s}.
	$$
\end{definition}

Note that $\GG(\xv)$ is symmetric and positive definite (see e.g. \cite[Proposition 2.6]{batenkov2018a}).

In a manner completely analogous to \prettyref{def:single.cluster}, we define the notion of a clustered configuration appropriate for this setting.

\begin{definition}[Multi-cluster configuration, real line]\label{def:partial.cluster.prolate}
  The node set $\xv=\left\{ x_1,\dots,x_s\right\} \subset \RR$ is said
  to form a $\left(\Delta,\theta,s,\ell,\tau\right)_{\RR}$-clustered
  configuration for some $\Delta>0$, $1\leq \ell\leq s$,
  $\ell-1\leq\tau$ and $\theta > 0$, if for each
  $x_j$, there exist at most $\ell$ distinct nodes
  $$
  \xv^{(j)}=\{x_{j,k}\}_{k=1,\dots,r_j} \subset \xv,\;1\leq
  r_j\leq\ell,\quad x_{j,1}\equiv x_j,
  $$
  such that the following conditions are satisfied:
  \begin{enumerate}
  \item For any $y\in\xv^{(j)}\setminus\{x_j\}$, we have
    $$
    	\Delta\leq \left|y-x_j\right| \leq \tau \Delta.
    $$
  \item For any $y\in\xv \setminus\xv^{(j)}$, we have
    $$
    \left|y-x_j\right| \geq \theta. 
    $$
  \end{enumerate}
\end{definition}

The next theorem is a direct corollary of
\prettyref{thm:main-theorem}, and it should be compared to
\eqref{eq:slepian-bound}.

\begin{theorem}\label{thm:prolate}
  There exist absolute constants $\Cl{multi.cluster.prolate.lower}$,
  $\Cl{multi.cluster.lower.O}=16\pi e$ and constants
  $\Cl{multi.cluster.O.theta}(\ell)$,
  $\Cl{multi.cluster.O.delta}(\ell)$, such that for any $\xv$ forming
  a $\left(\Delta,\theta,s,\ell,\tau\right)_{\RR}$-clustered
  configuration with $\theta \geq \Cr{multi.cluster.O.theta} s$ and
  $s\tau\Delta \leq \Cr{multi.cluster.O.delta}$,
  \begin{align}
      \label{eq.multi.cluster.bound.prolate}
      \lambda_{\min}\left(\GG\left(\xv\right)\right) &\geq
			\Cr{multi.cluster.prolate.lower}\left(\frac{\Delta}{\Cr{multi.cluster.lower.O}} \right)^{2(\ell-1)}.
  \end{align}
\end{theorem}

\begin{remark}
  Frequently the definition of the prolate matrix contains an
  additional bandwidth parameter $\Omega>0$, so that the inner
  products are considered in an interval $[-\Omega,\Omega]$
  \cite{batenkov2018a}. In the present paper we do not lose any
  generality by rescaling $\Omega$ to 1.
\end{remark}

\section{Prior art}\label{sec:comparison}

In this section only, $c,c_1,\dots,c',\dots,C,\dots$ denote generic
constants which might be different in different formulas, and which do
not depend on $N,\Delta$.

Let us start with the setting $\ell=s$. Recalling \prettyref{def:prolate},
we have, as $N\to\infty$, that
$$
(2N)^{-1}\sigma_{\min}^2 \left( \VV_N(\xv /N) \right) \to \lambda_{\min}
\left(\GG(\xv)\right).$$
In the equispaced setting $x_{j+1} = x_j + \Delta$,
$j=1,\dots,s-1$,
the matrix $\GG$ is precisely the ``prolate matrix''
\cite{slepian_prolate_1978,	varah_prolate_1993}, and it holds that
\begin{equation}\label{eq:slepian-bound}
	\begin{split}
		\lambda_{\min}(\GG(\xv)) &= C_{EQ}(s) \Delta^{2s-2}\left\{ 1 + O(\Delta) \right\}, \quad \Delta \ll 1;\\
		C_{EQ}(s)&:=\frac{2^{2 s-2}}{(2 s-1){2s-2 \choose s-1}^{3}}
		\asymp_s \left(1\over 4\right)^{2s-2}.
	\end{split}
\end{equation}

Here $\asymp_s$ means ``up to polynomial in $s$ and $1/s$
factors''. This gives
$$
\sigma_{\min}(\VV_N) \asymp_s \sqrt{N} \biggl(\frac{N\Delta}{4}
\biggr)^{s-1}\left\{1+O(\sqrt{N\Delta})\right\}.
$$

Non-asymptotic bounds for the case of node configurations $\xv$ with minimal
separation of at least $\Delta$ are available as well. An explicit construction
in \cite[Proposition 3]{li_stable_2017} (following
\cite{donoho_superresolution_1992}) gives
$$
\min_{\xv}\sigma_{\min}(\VV_N(\xv)) \lessapprox_s \sqrt{N}\biggl(\frac{N\Delta}{2}\biggr)^{s-1},\quad N\Delta < \frac{2\pi}{{C}_{LL}(s)\sqrt{N}},\quad {C}_{LL}(s) =  2\pi \sum_{j=0}^{s-1} {s-1 \choose j} \frac{j^s}{s!}.
$$

For a single cluster setting $\ell=s$, \prettyref{thm:single.cluster}
has been proven in \cite[Example 4.8]{kunis2019a} with the better
constant $\Cr{single.cluster.lower}=4\pi e$ (in the earlier work
\cite{li_stable_2017} this constant was not explicit). The reduction
in the tightness of constant in our work might be related to the fact
that our constant is also valid for all the singular values, see
\prettyref{thm:single.cluster.all} below.

Turning to the more general case $\ell\leq s$ and cluster separation
$\theta$ (as in \prettyref{def:partial.cluster}), in
\cite{li_stable_2017} and later \cite[Corollary 4.2]{kunis2019a} it
was shown that
$$
\sigma_{\min}(\VV_N(\xv)) \geq \frac{5}{9} \sqrt{N} \biggl(\frac{N\Delta}{4\pi e}\biggr)^{\ell-1},\quad N\Delta < 1.
$$
However, the above holds under the condition $N\theta >
c_{\gamma} (N\Delta)^{-\gamma}$ with any $\gamma>0$ and $c_\gamma \to
\infty$ as $\gamma\to 0$. For fixed $N,\theta$, this prevents
$\Delta\to 0$ in order for the bound to continue to hold. In contrast,
assuming $N\theta > c'$ and $N\Delta < c''$ with $c',c''$ depending
only on $s,\ell$, it was shown in \cite[Theorem 2.3, Corollary 2.1]{batenkov2019}
that for all $\xv$ satisfying the clustering geometric assumptions, we have
\begin{equation}\label{eq:orth.paper.bound}
	C'\sqrt{N} (N\Delta)^{\ell-1} \leq \sigma_{\min}(\VV_N(\xv)) \leq \frac{1}{2}\sqrt{N\ell e}  (\tau N\Delta)^{\ell-1}.
\end{equation}
Here $\tau \geq (\ell-1)$ is a uniformity parameter, controlling the
overall extent of any cluster (see
\prettyref{def:partial.cluster}). However, the constant $C'$ was not
explicit. In \cite{batenkov2018a} the same scaling for the lower bound
was shown with $C'=\frac{\sqrt{\pi/2}}{(s\sqrt{2\pi})^{2s-1}}$, albeit
under the additional assumption that $x_j \in \frac{\pi}{2s^2}(-1,1]$
for all $j=1,\dots, s$.

More details on the above developments are available in \cite[Section
1.4]{batenkov2019}, \cite[Examples 4.7,4.8]{kunis2019a}, \cite[Remarks
3.5,3.7]{batenkov2018a} and \cite[Remark 4]{li_stable_2017}.

Our single-exponential in $\ell$ bound of \prettyref{thm:main-theorem}, which
holds for a fixed $N$ and \emph{all} sufficiently small
$\Delta$, thus provides an improvement upon the above mentioned results (except, possibly, for the value of the absolute constant $\Cr{multi.cluster.lower.in}$, see \prettyref{sec:numerics}).

\section{Numerical experiments}\label{sec:numerics}

In this section we estimate the exponential dependence of $\sigma_{\min}(\VV_N(\xv))$ on $\ell$ numerically, by computing
$$
\Lambda(\xv,N):= \sigma_{\min}(\VV_N(\xv)/(\sqrt{N}(N\Delta)^{\ell-1}).
$$

Varying $\ell,\tau$ and $N,\Delta$ fixed, we expect that
$$
(1/c_2)^{\ell-1} \lessapprox \Lambda(\xv,N) \lessapprox (1/c_2')^{\ell-1}.
$$ 

Our theoretical results indicate that the above holds with $c_2 \leq 16\pi e$  and $c_2' \geq {1/\tau}$ (see resp. \eqref{eq:last-with-16pe} and \eqref{eq:orth.paper.bound}).

As can be seen from \prettyref{fig:single-cluster-nd-vary}, both the upper and lower bounds are correct, although the corresponding constants $16\pi e$ and $1/\tau$ are not tight.

\begin{figure}
	\centering
	\includegraphics[width=0.8\linewidth]{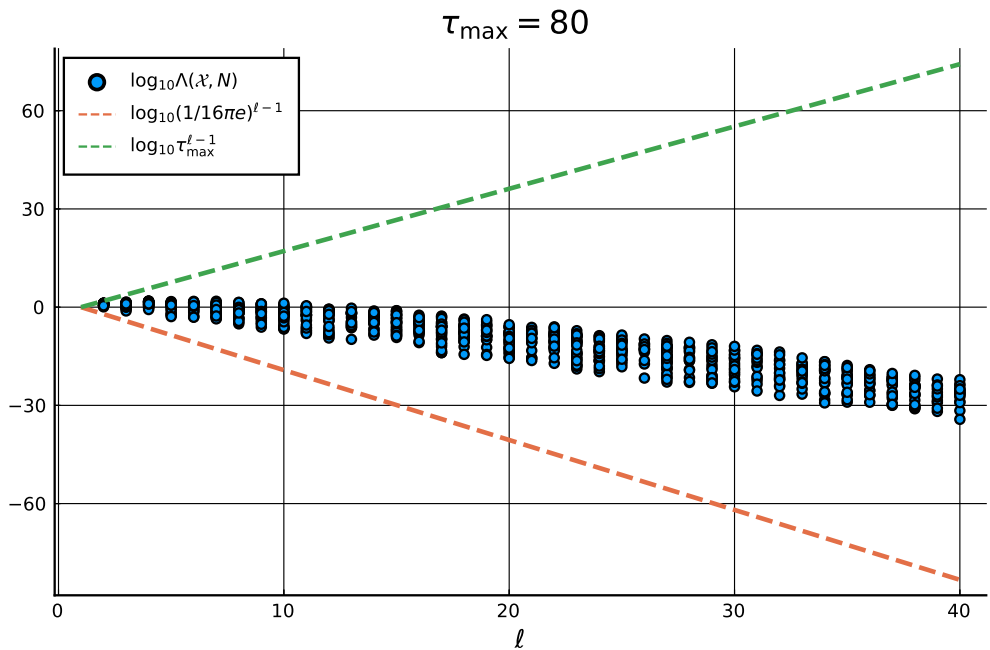}
	\caption{Single cluster - dependence of $\sigma_{\min}$ on
		$\ell$. We plot $\ell$ vs
		$\log_{10}\Lambda(\xv,N)$. $N$ is varying between 60 and 300, $\Delta$ is varying between $10^{-25}$ and $10^{-4}$, while $\ell$ varies from 2 to 40 and $\tau$ varies from $\ell-1$ to $\tau_{\max}=80$. The lower and upper bounds are shown as dashed lines.}
	\label{fig:single-cluster-nd-vary}
\end{figure}

All numerical tests were performed in arbitrary precision arithmetic.

\section{Discussion}

It is an interesting open question whether a bound of the type
\eqref{eq:main-lower-bound} should hold in the multi-cluster geometry,
for $N\theta > c', N\Delta < c''$ where $c',c''$ do not depend on $s$
and with no essential further restrictions on $\xv$. If this is the
case, then it is plausible that the super-resolution problem for a
practically infinite spike train ($s \gg 1$) with small sub-Rayleigh
clusters (a model analogous to Donoho's Rayleigh regular measures,
\cite{donoho_superresolution_1992}) can be essentially decoupled into
treating each cluster separately.

There is a room for further refinement regarding the bounds
themselves, as there is a relatively large gap in the
constants between the upper bound in \eqref{eq:orth.paper.bound} and
\eqref{eq:main-lower-bound}.

\section{Proofs}

\subsection{Preliminaries on exponential sums}
We review some preliminary results about exponential sums and their implications to the problem at hand.

\begin{definition}
	Given a vector $\cv\in\CC^\ell$ and $\xv=\{ x_1,\ldots,x_\ell\}\subset \RR$, we define the {\it exponential sum}
	$$
	P(t) = P_{\cv,\xv}(t):=\sum_{j=1}^\ell c_j e^{\imath t x_j}.
	$$
	The number of nonzero $c_j$'s is called the \emph{degree} of $P$. The set of all exponential sums of degree at most $\ell$ is denoted by $\exs$.
\end{definition}

\begin{remark}
	Our definition of exponential sums covers the case of purely imaginary exponents only to simplify the presentation. More general results for arbitrary complex exponents are available in e.g. \cite{nazarov_local_1994, erdelyi2017}.
\end{remark}

We denote by $\mu$ the Lebesgue measure on $\RR$.

%
Given an interval $I$ and a (complex valued) continuous function $f\in C(I)$,  $1 \le p \le \infty$, we denote
$$
\|f\|_{L^p(I)}:=
\begin{cases}
	\biggl(\frac{1}{\mu(I)}\int_{I} |f(t)|^p dt\biggr)^{1/p},& 1 \le p < \infty;\\
	\sup_{t\in I} |f(t)|,& p=\infty.
\end{cases}
$$

Exponential sums satisfy many classical inequalities from
approximation theory. In particular, we have the following estimates.

\begin{proposition}[Turan's inequality]
	Let $P\in\exs$, and let $\Omega \subset I$ be intervals with positive Lebesgue measure. Then
	$$
	\|P\|_{L^{\infty}(I)} \leqslant\left(\frac{4e \mu(I)}{\mu(\Omega)}\right)^{\ell-1} \|P\|_{L^{\infty}(\Omega)}.
	$$
\end{proposition}
\begin{proof}
  See a review on Turan's lemma on p.7 of \cite{nazarov_local_1994} (and Turan's original result 
  \cite{turan1953neue}, in German).
  
\end{proof}

\begin{proposition}[Nikolskii-type inequality]\label{prop:nikolskii}
	Let $P\in\exs$, then
	$$
	\|P\|_{L^{p}[0,1]} \leqslant\left(\frac{\pi \ell}{2}\right)^{2 / q-2 / p}\|P\|_{L^{q}[0,1]}, \quad 0<q<p \leqslant \infty, \quad q \leqslant 2.
	$$
\end{proposition}
\begin{proof}
	This is Theorem 2.5 in \cite{erdelyi2017}.
\end{proof}

Applying the above with $I=\left[0,\frac{4\pi}{\Delta}\right]$, $\Omega=\left[0,N\right]$, $p=\infty$ and $q=2$ yields the following.

\begin{corollary}\label{cor.turan}
	Let $\xv$ form an $(\Delta, \ell, \tau)_{\TT}$-clustered configuration
	and let  $N \leq \frac{4\pi}{\Delta}$, then for any $\cv \in \CC^\ell$
	\begin{equation}\label{eq:cor.turan}
		\|P_{\cv,\xv}\|_{L^2\left(\left[0,N\right]\right)} \geq \frac{2}{\pi \ell} \biggl(\frac{N\Delta}{16\pi e}\biggr)^{\ell-1} 
		\|P_{\cv,\xv}\|_{L^2\left(\left[0,\frac{4\pi}{\Delta}\right]\right)}.
	\end{equation}
\end{corollary}
\begin{proof}
	Indeed, we have
	\begin{align*}
		\|P_{\cv,\xv}\|_{L^2\left(\left[0,\frac{4\pi}{\Delta}\right]\right)} &\leq \|P_{\cv,\xv}\|_{L^{\infty}\left(\left[0,\frac{4\pi}{\Delta}\right]\right)} 
		\leq \biggl(\frac{4e\cdot 4\pi}{N\Delta}\biggr)^{\ell-1} \|P_{\cv,\xv}\|_{L^{\infty}[0,N]} 
		\leq \frac{\pi\ell}{2}\biggl(\frac{16\pi e}{N\Delta}\biggr)^{\ell-1}\|P_{\cv,\xv}\|_{L^2[0,N]}.
	\end{align*}
\end{proof}

Now consider an exponential sum $\Pcx$ where $\Delta$ is the minimal
separation of the nodes in $\xv$.  The next result states that for
intervals $I$ with length of the order of $\frac{1}{\sep}$ or more,
the coefficients norm $\|\cv\|_{2}$ and $\|\Pcx\|_{L^2(I)}$ are
related by an absolute constant. This is contrary to the case when the
length of $I$ is smaller than $\frac{1}{\sep}$, in which case the
constant will depend on $\sep \cdot \mu(I)$ and $\ell$, as we will
show below.
\begin{proposition}\label{prop.salem}
	Let $\xv$ form an $(\Delta,\ell,\tau)_{\TT}$-clustered configuration and let $\cv \in \CC^\ell$. Then, there
	exists an absolute constant $\Cl{salem}$ such that 
	$$
	\|P_{\cv,\xv}\|^2_{L^2\left(\left[0,\frac{4\pi}{\Delta}\right]\right)} \geq \Cr{salem}\|\cv\|_2^2.
	$$
\end{proposition}
\begin{proof}
	It directly follows from \cite[Vol.I, Chapter V,
	Th. 9.1]{zygmund1959} \footnote{Nazarov calls this type of
		inequality Salem's Inequality, see \cite[page
		8]{nazarov_local_1994}.}  stating that for an interval $I$ such
	that $\mu(I)=\frac{2\pi(1+\delta)}{\sep}$, $\delta>0$, there exists
	an absolute constant $C$ such that
	$$\|\cv\|_2^2 \le C(1-\delta^{-1})\|\Pcx\|^2_{L^2(I)}.$$   	
\end{proof}

Finally we require the following Bernstein type inequality bounding
the maximum absolute value of the derivative of an exponential sum on
$[0,1]$, by its maximum absolute value on $[0,1]$. See proof in
\cite[Theorem 2.20]{erdelyi2017}

\begin{proposition}\label{prop:bernstein}
	Let $\Pcx$ be an exponential sum, $\cv\in\CC^\ell$ and $\xv=\{ x_1,\ldots,x_\ell\}\subset \RR$, then
	$$\|\Pcx'\|_{L^{\infty}([0,1])} \le \Cl{bernstein}\left(108\ell^5 +\sum_{j=1}^{\ell} x_j^2\right)^\frac{1}{2}\|\Pcx\|_{L^{\infty}([0,1])}.$$ 
\end{proposition}

\subsection{Proof of Theorem \ref{thm:single.cluster}}\label{sec:proof-single-cluster}
Let $\xv$ form $(\Delta,\ell,\tau)_{\TT}$-clustered configuration as in
\prettyref{thm:single.cluster} and assume, without loss of generality,
that $\xv$ is centered around the origin, i.e.
$\min_j x_j+ \max_j x_j=0$, which implies that
$\xv\subset[-\tau\Delta/2,\tau\Delta/2]$.
 
Let
$$
\|P\|_{2,N}:=\left(\sum_{k=0}^N \left|P(k)\right|^2\right)^{1/2}.
$$
Then
\begin{equation}\label{eq.sum.sing}
	\sigma_{\min}(\VV_N(\xv))= \min_{\|\cv\|_2=1}\|P_{\cv,\xv}\|_{2,N}.
\end{equation}

Fix some $\cv \in \CC^\ell$ such that $\|\cv\|_2 = 1$ and $N \le \frac{2 \pi}{\tau \Delta}\le \frac{4 \pi}{\Delta}$.
Combining Corollary \ref{cor.turan} and Proposition \ref{prop.salem}
we obtain
\begin{equation}\label{eq.single.cluster.norm}
   \|P_{\cv,\xv}\|_{L^2\left(\left[0,N\right]\right)} \geq \frac{2}{\pi \ell}\biggl(\frac{N\Delta}{16\pi e}\biggr)^{\ell-1} 
   \|P_{\cv,\xv}\|_{L^2\left(\left[0,\frac{4\pi}{\Delta}\right]\right)}\ge \frac{2}{\pi \ell} \biggl(\frac{N\Delta}{16\pi e}\biggr)^{\ell-1}\sqrt{\Cr{salem}}\|\cv\|_2= \frac{\Cl{temp.1}}{\ell}
   \biggl(\frac{N\Delta}{16\pi e}\biggr)^{\ell-1}.
\end{equation}

At this point, we ``almost'' have the required result, 
what is left is to relate $\|P\|_{2,N}$ and the norm $\|P_{\cv,\xv}\|_{L^2\left(\left[0,N\right]\right)}$,
as follows.

Define
$$
Q_{\cv,\xv,N}(u):= \sum_{j=1}^\ell c_j e^{\imath N u x_j}.
$$
Then $Q_{\cv,\xv,N}(u)=P_{\cv,\xv}(Nu)$ and
$$
\|P_{\cv,\xv}\|_{L^2\left(\left[0,N\right]\right)} = \|Q_{\cv,\xv,N}\|_{L^2\left([0,1]\right)}.
$$

Put
$$
T_{\cv,\xv,N}(u):=Q_{\cv,\xv,N}(u) \widebar{Q}_{\cv,\xv,N}(u) = \left|Q_{\cv,\xv,N}(u)\right|^2.
$$

We have that
\begin{align}
  \label{eq:TandQ}
  \begin{split}
  \|P_{\cv,\xv}\|_{L^2\left(\left[0,N\right]\right)}^2 &= \|T_{\cv,\xv,N} \|_{L^1\left([0,1]\right)}=\int_{0}^1 T_{\cv,\xv,N}(u)du,\\
  \|P_{\cv,\xv}\|_{2,N}^2&=\sum_{k=0}^N T_{\cv,\xv,N}\left({k\over N}\right).
\end{split}
\end{align}

$T_{\cv,\xv,N}= \sum_{j=1}^w b_j e^{iu\lambda_j}$ is an exponential sum of maximal degree $w:=\ell^2-\ell+1$, with the frequencies
satisfying $|\lambda_j|\leq {\tau N\Delta}$. Consequently by \prettyref{prop:bernstein}   
\begin{equation}\label{eq.bernstein.1}
	  	\|T_{\cv,\xv,N}'\|_{L^{\infty}([0,1])} \leq \Cr{bernstein}(108w^{5}+w(\tau N \Delta)^2)^{1/2} \| T_{\cv,\xv,N}\|_{L^{\infty}([0,1])}.
\end{equation}

Approximating the integral by a Riemann sum and using equation \eqref{eq.bernstein.1} we have
\begin{align*}
  \biggl| \int_{0}^1 T_{\cv,\xv,N}(u)du - {1\over N}\sum_{k=0}^N T_{\cv,\xv,N}\left({k\over N}\right)\biggr| 
  &\le {1\over 2N}\|T_{\cv,\xv,N}'\|_{L^{\infty}([0,1])}+{1\over N}T_{\cv,\xv,N}(0) \\
  &\leq {1\over{2N}} \Cr{bernstein}(108w^{5}+w(\tau N \Delta)^2)^{1/2} \| T_{\cv,\xv,N}\|_{L^{\infty}([0,1])} \\
  &\;\;\;\;\;+ {1\over N}T_{\cv,\xv,N}(0)\\ 
  &\leq {1\over N}\biggl({1\over 2} \Cr{bernstein}(108w^{5}+w(\tau N \Delta)^2)^{1/2} + 1 \biggr) \| T_{\cv,\xv,N}\|_{L^{\infty}([0,1])}.
\end{align*}

By assumption $N\le \frac{2 \pi}{\tau \Delta}$, therefore for
an absolute constant $\Cl{AAA}$ 
\begin{align}\label{eq:temp1}
  \biggl| \int_{0}^1 T_{\cv,\xv,N}(u)du - {1\over N}\sum_{k=0}^N T_{\cv,\xv,N}\left({k\over N}\right)\biggr| &\leq 
  	\frac{\Cr{AAA} \ell^{5}}{N} \| T_{\cv,\xv,N}\|_{L^{\infty}([0,1])}.
\end{align}

In addition, using \prettyref{prop:nikolskii} with $p=\infty,q=1$ and $n=w$, we
have
\begin{equation}\label{eq.nikolskii}
  \|T_{\cv,\xv,N}\|_{L^{\infty}([0,1])} \leq \left({\pi w}\over 2\right)^2 \|T_{\cv,\xv,N} \|_{L^1([0,1])}.
\end{equation}

Therefore, substituting \eqref{eq.nikolskii} into  \eqref{eq:temp1}, we obtain
\begin{align}\label{eq.temp.2}
   \biggl| \|T_{\cv,\xv,N} \|_{L^1([0,1])} - {1\over N}\sum_{k=0}^N T_{\cv,\xv,N}\left({k\over N}\right)\biggr|  
   \leq \frac{\Cr{AAA} \ell^{5}}{N} \left({\pi w}\over 2\right)^2 \|T_{\cv,\xv,N} \|_{L^1([0,1])}\leq \Cl{BBB}\frac{\ell^9}{N} \|T_{\cv,\xv,N} \|_{L^1([0,1])}
\end{align}

For $N\ge2\Cr{BBB} \ell^{9}$, we get from \eqref{eq.temp.2} that
\begin{align*}
  {1\over N}\sum_{k=0}^N T_{\cv,\xv,N}\left({k\over N}\right) &\geq\frac{1}{2} \|T_{\cv,\xv,N} \|_{L^1([0,1])}.
\end{align*}

By \eqref{eq:TandQ} we conclude that
\begin{align}\label{eq.integer.cont.norm.const}
   \|P_{\cv,\xv}\|_{2,N}^2 \geq \frac{N}{2}\|P_{\cv,\xv}\|^2_{L^2(\left[0,N\right])}.
\end{align}

Finally substituting \eqref{eq.integer.cont.norm.const} into \eqref{eq.single.cluster.norm} we get that 
\begin{equation}\label{eq:last-with-16pe}
\|P_{\cv,\xv}\|_{2,N} \ge
\frac{\Cr{temp.1}}{\ell}\sqrt{\frac{N}{2}} \biggl(\frac{N\Delta}{16\pi
  e}\biggr)^{\ell-1}.
\end{equation} 

Note that for all $\ell \geq 1$ we have $2^{\ell-1} \geq \ell$.
Since $\cv$ was arbitrary, using the relation \eqref{eq.sum.sing}
completes the proof of Theorem \ref{thm:single.cluster} with
$\Cr{single.cluster.N}(\ell)=2\Cr{BBB} \ell^{9}$,
$\Cr{single.cluster.lower}=32\pi e$ and
$\Cr{single.cluster.mult}=\Cr{temp.1}/\sqrt{2}$. \qed

\subsection{Proof of  \prettyref{thm:main-theorem}}\label{sec:proof-main-thm}
Let $\xv$ form a $\left(\Delta,\theta,s,\ell,\tau\right)_{\TT}$-clustered
configuration. Then there exists an $M$-partition
$\xv=\biguplus_{j=1}^M \mathcal{C}^{(j)}$ such that for each
$j\in\left\{1,\dots,M\right\}$:
\begin{enumerate}
\item $\mathcal{C}^{(j)}$ form a $(\Delta,\ell^{(j)},\tau)_{\TT}$-cluster
  according to \prettyref{def:single.cluster}, where $\ell^{(j)}\leq \ell$;
\item
  $ \dst(x,y) \ge \theta,\quad \forall x\in\mathcal{C}^{(j)},\;\forall
  y \in \xv\setminus\mathcal{C}^{(j)}$.
\end{enumerate}

By \eqref{eq:single-min-lb}, we
have that for each $j=1,\dots,M$
\begin{equation}\label{eq:single.clust.foreach}
\sigma_{\min}(\VV_N(\mathcal{C}^{(j)})) \ge \Cr{single.cluster.mult}
\sqrt{N}\biggl(\frac{N \Delta}{\Cr{single.cluster.lower}}\biggr)^{\ell
  -1}, \quad \Cr{single.cluster.N}(\ell)\le N \le \frac{2\pi}{\tau
  \Delta}.
\end{equation}

We now apply Theorem 2.2 in \cite{batenkov2019}, whose reduced version
reads as follows.

\begin{proposition}\label{prop:singvals-from-laa}
  Let $\xv$ form a $\left(\Delta,\theta,s,\ell,\tau\right)_{\TT}$-clustered
  configuration. Then there exist constants
  $\Cl{multi.cluster.N.theta.union}(\ell),\Cl{multi.cluster.N.delta.union}(\ell)$,
  depending only on $\ell$, such that whenever
  \begin{equation}
    \label{eq:orth.condition}
    \frac{\Cr{multi.cluster.N.theta.union}s}{\theta} \leq N \leq \frac{\Cr{multi.cluster.N.delta.union}}{s\tau\Delta},
  \end{equation}
  we have
  \begin{equation*}
    \sigma_{\min}(\VV_N(\xv)) \geq \frac{1}{2}\min_{j=1,\dots,M}\sigma_{\min}(\VV_N(\mathcal{C}^{(j)})).
  \end{equation*}
\end{proposition}

Since $\theta\leq\pi$ and $s\geq 1$, both lower bounds on $N$ in
\eqref{eq:single.clust.foreach} and \eqref{eq:orth.condition} are
satisfied whenever
$N\theta \geq s \max\left(\Cr{single.cluster.N}(\ell)\pi,
  \Cr{multi.cluster.N.theta.union} \right)$. On the other hand,
$N\tau\Delta < \min(2\pi,\Cr{multi.cluster.N.delta.union})$ implies
the corresponding upper bounds on $N$. This completes the proof of
\prettyref{thm:main-theorem} with
$\Cr{multi.cluster.mult} = \Cr{single.cluster.mult}/2$,
$\Cr{multi.cluster.N.theta}(\ell)=\max(\Cr{multi.cluster.N.theta.union},\Cr{single.cluster.N}(\ell)\pi)$,
$\Cr{multi.cluster.N.delta}(\ell) =
\min(2\pi,\Cr{multi.cluster.N.delta.union}(\ell))$ and
$\Cr{multi.cluster.lower.in} = \Cr{single.cluster.lower}=32\pi e$. \qed

\subsection{Proof of Theorem \ref{thm:prolate}}\label{sec.prolate.proof}
\begin{proof}
  Let $\theta \geq \frac{\Cr{multi.cluster.N.theta}s}{2}$ and
  $s\tau\Delta \leq \frac{\Cr{multi.cluster.N.delta}}{2}$, where the
  constants
  $\Cr{multi.cluster.N.theta}=\Cr{multi.cluster.N.theta}(\ell)$,
  $\Cr{multi.cluster.N.delta}=\Cr{multi.cluster.N.delta}(\ell)$, are
  the same as in Theorem \ref{thm:main-theorem}. Now let $\xv$ form a
  $\left(\Delta,\theta,s,\ell,\tau\right)_{\RR}$-clustered
  configuration.

  For any $N\in\NN$ and
  $j\in\{1,\ldots,s\}$ put $\xi_j = \xi_{j,N}:=\frac{x_j}{N}$
  and $\xiv=\xiv^{(N)}:= \{\xi_1,\ldots,\xi_s\}$, and define the
  following shifted in frequency and normalized Vandermonde like
  matrix
	\begin{equation*}\label{eq.shifted.vand}
		\widetilde{\VV}_N\left(\xiv\right):=\frac{1}{\sqrt{2N}} \bigl[ \exp\left(\imath k \xi_j\right)\bigr]_{k=-N,\dots,N}^{j=1,\dots,s}\;.
	\end{equation*}
	We have 
	\begin{equation*}
  		\frac{1}{2} \int_{-1}^1 \exp(\ii\w t) d\w
  			=\lim_{N\to\infty}\frac{1}{2N}\sum_{k=-N}^N \exp\left( \imath \frac{k}{N} t\right).
	\end{equation*}
	Consequently $\GG(\xv) = \lim_{N\to\infty} \widetilde{\VV}_N(\xiv)^H \widetilde{\VV}_N(\xiv)$,
	and so by continuity of eigenvalues \cite[Section 2.4.9]{horn_matrix_2012} we have that
	\begin{equation}\label{eq.V.G}
		\lambda_{\min}\left(\GG(\xv)\right) = \lim_{N\to\infty} \lambda_{\min}\left( \widetilde{\VV}_N(\xiv)^H \widetilde{\VV}_N(\xiv)\right)
		=\lim_{N\to\infty} \sigma^2_{\min}\left(\widetilde{\VV}_N(\xiv)\right).
	\end{equation}
	
	For $N$ large enough we have $\{\xi_{1,N},\ldots,\xi_{s,N}\} \subset (-\pi, \pi]$ and we can write $\widetilde{\VV}_N\left(\xiv\right)$ as
	\begin{equation}\label{eq.V.Vtilde}
		\widetilde{\VV}_N\left(\xiv\right) = \frac{1}{\sqrt{2N}} \VV_{2N}(\xiv)\cdot \diag\left(e^{-\ii N \xi_1},\ldots,e^{-\ii N \xi_s}\right),
	\end{equation}
	where $\diag\left(e^{-\ii N \xi_1},\ldots,e^{-\ii N \xi_s}\right)$ is the $s \times s$ 
	diagonal matrix with $\left(e^{-\ii N \xi_1},\ldots,e^{-\ii N \xi_s}\right)$ as its main diagonal. 
	By \eqref{eq.V.Vtilde} clearly 
	\begin{equation}\label{eq.prolate.sig.Vtilde.V}
		\sigma_{\min}\left(\widetilde{\VV}_N\left(\xiv\right)\right)= \frac{1}{\sqrt{2N}} \sigma_{\min}(\VV_{2N}(\xiv)).
	\end{equation}
	One can validate that for each $N$, $\{\xi_{j,N}\}_{j=1}^s$ form a
        $\left(\frac{\Delta}{N},\frac{\theta}{N},s,\ell,\tau\right)_{\TT}$-clustered
        configuration and, on the other hand, the assumptions
        $\theta \geq \frac{\Cr{multi.cluster.N.theta}s}{2}$ and
        $s\tau\Delta \leq \frac{\Cr{multi.cluster.N.delta}}{2}$ imply
        that
        $\frac{\Cr{multi.cluster.N.theta}s}{\left(\frac{\theta}{N}\right)}
        \le 2N \le
        \frac{\Cr{multi.cluster.N.delta}}{s\tau\left(\frac{\Delta}{N}\right)}$.
        Now we apply Theorem \ref{thm:main-theorem} and obtain
        $\sigma_{\min}\left(\VV_{2N}(\xiv)\right) \ge \Cr{multi.cluster.mult}\sqrt{2N}
        \left(\frac{2\Delta}{\Cr{multi.cluster.lower.in}}
        \right)^{\ell-1}$ and therefore using
        \eqref{eq.prolate.sig.Vtilde.V} we have
	\begin{equation*}
          \sigma_{\min}\left(\widetilde{\VV}_N\left(\xiv\right)\right) \ge \Cr{multi.cluster.mult} \left(\frac{2\Delta}{\Cr{multi.cluster.lower.in}} \right)^{\ell-1}.
	\end{equation*}
	Finally using \eqref{eq.V.G} we get that
        $\lambda_{\min}\left(\GG(\xv)\right)\ge
        \Cr{multi.cluster.mult}^2
        \left(\frac{2\Delta}{\Cr{multi.cluster.lower.in}}
        \right)^{2\ell-2}$.  This proves Theorem \ref{thm:prolate}
        with $\Cr{multi.cluster.prolate.lower}=\Cr{multi.cluster.mult}^2$,
        $\Cr{multi.cluster.lower.O}=\frac{\Cr{multi.cluster.lower.in}}{2}=16\pi e$,
        $\Cr{multi.cluster.O.theta}=\frac{\Cr{multi.cluster.N.theta}}{2}$
        and
        $\Cr{multi.cluster.O.delta} =
        \frac{\Cr{multi.cluster.N.delta}}{2}$.
      \end{proof}

\section{Entire spectrum}\label{sec:entire-spectrum}
      As mentioned in the Introduction, our proofs can be extended to
      provide scaling for \emph{all the singular values} of $\VV_N$
      (resp. eigenvalues of $\GG$.)
      
      For a single cluster, we have the following more general result
      from which \prettyref{thm:single.cluster} immediately follows as
      a corollary.

\begin{theorem}\label{thm:single.cluster.all}
  Let $\xv$ form a $(\Delta,\ell,\tau)_{\TT}$-clustered
  configuration. Denote the singular values of $\VV_N(\xv)$ by
  $$
  \sigma_1 \geq \sigma_2 \geq \dots \geq \sigma_\ell\equiv\sigma_{\min}.
  $$
  Then for any $N$ satisfying
  $\Cr{single.cluster.N}(\ell)\le N \le \frac{2\pi}{\tau \Delta}$,
  there holds
\begin{equation}
  \label{eq:single-cluster-all-bound}
  \sigma_m(\VV_N(\xv)) \geq   \Cr{single.cluster.mult}\sqrt{N}\biggl(\frac{N
    \Delta}{\Cr{single.cluster.lower}}\biggr)^{m -1},\qquad m=1,2,\dots,\ell.
\end{equation}
All the constants are the same as in \prettyref{thm:single.cluster}.
\end{theorem}

\begin{proof}[Proof outline]
  Fix $m=1,2,\dots,\ell$, and let $\cv\in\CC^m$ with $\|\cv\|_2=1$ be
  arbitrary.  Furthermore, denote by $\xv_m$ the ordered set
  $\{x_1,\dots,x_m\}\subseteq \xv$. By the Courant-Fischer minmax
  principle we have, extending \eqref{eq.sum.sing}, that
  \begin{equation}
    \label{eq:sv-via-Pnorm-all}
    \sigma_{m}(\VV_N(\xv)) \geq \min_{\cv\in\CC^m, \; \|\cv\|_2=1} \|P_{\cv,\xv_m}\|_{2,N},\qquad m=1,2,\dots,\ell.
  \end{equation}
  Now we can repeat the computation from Section
  \ref{sec:proof-single-cluster}, replacing $\ell$ with $m$ and $\xv$
  with $\xv_m$.
\end{proof}

In order to provide appropriate extensions of
\prettyref{thm:main-theorem} and \prettyref{thm:prolate}, recall the
construction of the $M$-partition of $\xv$ from
\prettyref{sec:proof-main-thm}. Now for each $m=1,2,\dots,\ell$ let
$q_m$ be the number of clusters among the $\mathcal{C}^{(j)}$ of
multiplicity at least $m$:
\begin{equation}\label{eq:qm-def}
  q_m:=\#\{1\leq j\leq M:  m \leq \ell^{(j)}\}.
\end{equation}

The extension of \prettyref{thm:main-theorem} to include all the
singular values is the following.

\begin{theorem}\label{thm:all-sing-v}
  Let $\xv$ form a $\left(\Delta,\theta,s,\ell,\tau\right)_{\TT}$ as
  in \prettyref{thm:main-theorem}. Then for $N$ as in
  \prettyref{thm:main-theorem}, for each $m=1,2,\dots,\ell$ there are
  precisely $q_m$ singular values of $\VV_N(\xv)$ bounded from below by
  \begin{align}
      \label{eq.multi.cluster.bound.all}
      \Cr{multi.cluster.mult}\sqrt{N} \left(\frac{N \Delta}{\Cr{multi.cluster.lower.in}} \right)^{m-1}.
  \end{align}

\end{theorem}

To prove this result, we repeat the proof from
\prettyref{sec:proof-main-thm}, replacing
\prettyref{prop:singvals-from-laa} with its ``full'' version from
\cite{batenkov2019} which reads as follows.

\begin{proposition}[Theorem 2.2 in \cite{batenkov2019}]
  Let $\xv$ form a
  $\left(\Delta,\theta,s,\ell,\tau\right)_{\TT}$-clustered
  configuration.  Let
  $\tilde{\sigma}_1\geq \tilde{\sigma}_2\geq\dots\geq\tilde{\sigma}_s$
  denote the union of all the singular values of the matrices
  $\VV_N(\mathcal{C}^{(j)})$ in non-increasing order, and
  $\sigma_1\geq\dots\geq \sigma_s$ denote the singular values of
  $\VV_N(\xv)$.  Then whenever \eqref{eq:orth.condition} holds, we
  have
  \begin{equation*}
    \sigma_j \geq \frac{1}{2}\tilde{\sigma}_j,\quad j=1,\dots,s.
  \end{equation*}
\end{proposition}

As for \prettyref{thm:prolate}, we can define the numbers $q_m$ in a
similar manner with respect to the clustered configurations on $\RR$,
and then we have the following.

\begin{theorem}\label{thm:all-eig-prol}
  For any $\xv$ forming
  a $\left(\Delta,\theta,s,\ell,\tau\right)_{\RR}$-clustered
  configuration as in \prettyref{thm:prolate}, for each
  $m=1,2,\dots,\ell$ there are precisely
  $q_m$ eigenvalues of $\GG\left(\xv\right)$ bounded from below by
  $$
  \Cr{multi.cluster.prolate.lower}\left(\frac{\Delta}{\Cr{multi.cluster.lower.O}} \right)^{2(m-1)}.
  $$
\end{theorem}

The proof is identical to that of \prettyref{thm:prolate}, noting 
that \eqref{eq.V.G} and \eqref{eq.prolate.sig.Vtilde.V} hold for all
the singular values, and using \prettyref{thm:all-sing-v} in place of \prettyref{thm:main-theorem}.

\bibliographystyle{plainurl}
\bibliography{bib}	
\end{document}